\documentclass[
a4paper % Adapt the display to A4 paper. 
,11pt % Font size.
,leqno % Equation numbers on the left.
]{scrartcl} % KOMA-script.

\usepackage{my-preamble-article} % Loads the required packages for articles. 
\usepackage{my-macros} % My macros. 

\usepackage[nostamp]{draftwatermark}
\SetWatermarkText{Work in progress}
\SetWatermarkColor[rgb]{1,0.4,0.4}
\SetWatermarkScale{1.5}

\DeclareMathOperator{\vnone}{V_\textup{\textrm{none}}}
\DeclareMathOperator{\vSteep}{V_\textup{\textrm{steep}}}
\DeclareMathOperator{\vGradual}{V_\textup{\textrm{gradual}}}
\DeclareMathOperator{\vinfty}{V_\textup{\textrm{diverg}}}
\DeclareMathOperator{\vstab}{V_\textup{\textrm{stab}}}
\DeclareMathOperator{\vcent}{V_\textup{\textrm{cent}}}

\title{Retracting fronts for the nonlinear complex heat equation}

\usepackage{authblk} % Présentation des auteurs en page de titre. 
\author{Guillaume \textsc{Réocreux}}
\author{Emmanuel \textsc{Risler}\footnote{\protect\url{\myWebPage}}}
\affil{\myAffiliation}

\begin{document}
\maketitle
\thispagestyle{empty}
\begin{abstract}
The \emph{nonlinear complex heat equation} $A_t=i\abs{A}^2A+A_{xx}$ was introduced by P.~Coullet and L.~Kramer as a model equation exhibiting travelling fronts induced by non-variational effects, called \emph{retracting fronts}. In this paper we study the existence of such fronts. They go by one-parameter families, bounded at one end by the slowest and ``steepest'' front among the family, a situation presenting striking analogies with front propagation into unstable states. 
\end{abstract}
\section{Introduction}
The \emph{nonlinear complex heat equation}
\begin{equation}
\label{nlche}
A_t=i\abs{A}^2A+A_{xx}
\,,
\end{equation}
where time variable $t$ and space variable $x$ are real and amplitude $A(x,t)$ is complex, is a model equation combining a purely dispersive nonlinearity and a purely diffusive coupling. It appears as a particular limit of complex Ginzburg--Landau equations, and displays the peculiar feature that there is no characteristic scale for the complex amplitude $A$ (that is, equation is scale-invariant, see \cref{scale_inv} below). 

Despite its simplicity, this equation displays a nontrivial dynamical behaviour. It was introduced by P.~Coullet and L.~Kramer \cite{CoulletKramer_retractingFronts_2004} as a model equation exhibiting travelling fronts induced by non-variational effects, called ``retracting fronts''. Indeed, consider an initial condition connecting the state $A=0$ to a homogeneous oscillatory state. The effect of the coupling term is to smooth the interface (without the nonlinear term one would observe a self-similar repair), but the nonlinear term generates a phase gradient that pushes the interface in favour of the zero-amplitude state. Numerically, the solution converges towards a travelling front that balances these effects of diffusion and nonlinearity \cite{CoulletKramer_retractingFronts_2004}. When a small perturbation is added, rending, on one hand the zero-amplitude state slightly unstable, and on the other hand homogeneous perturbation at a certain finite amplitude stable, then the presence of these retracting fronts induces spatio-temporal intermittency \cite{CoulletKramer_retractingFronts_2004}. As a matter of fact the initial motivation of Coullet and Kramer to introduce a model equation displaying this phenomenon was an experiment of Rayleigh-Bénard convection with rotation at small Prandtl numbers, for which intermittend convection without hysteresis was observed \cite{AhlersBajajPesch_RayleighBenardConvRotSmallPrandtlNumbers_2002}.

The aim of this work is to investigate the existence of these retracting fronts for the model equation \cref{nlche}. 
\section{Main result}
\label{sec:main_res}
The spatially homogeneous solutions of equation \cref{nlche} are the ``trivial'' solution
\[
A(x,t)\equiv 0
\,,
\]
and, for every pair $(a,\varphi)$ of real quantities, the spatially  homogeneous time-periodic solution
\begin{equation}
\label{spat_hom}
A(x,t)= a \exp\bigl(i a^2t+\varphi\bigr)
\,.
\end{equation}
A larger class of particular solutions are uniformly translating and rotating solutions. For sake of clarity, let us state a formal definition. 
\begin{definition}[uniformly translating and rotating solution]
A solution $(x,t)\mapsto A(x,t)$ of equation \cref{nlche} is called a \emph{uniformly translating and rotating solution} if there exist real quantities $v$ and $\omega$ and a smooth function $B:\rr\rightarrow\cc$ such that, for every $(x,t)$ in $\rr^2$, 
\begin{equation}
\label{ansatz}
A(x,t)=B(x-vt)e^{i\omega t} 
\,.
\end{equation}
\end{definition}
Our target is the more specific class of particular solutions defined immediately below, and illustrated on \cref{fig:shape_front}.
\begin{figure}[!htbp]
	\centering
    \includegraphics[width=\textwidth]{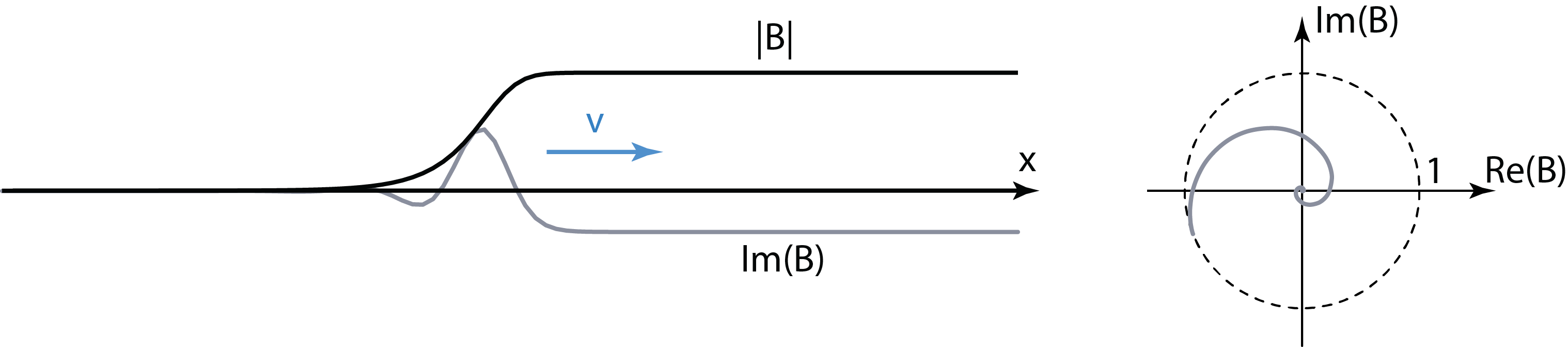}
    \caption{Numerical computation of ``the'' retracting front (courtesy of Pierre Coullet).}
    \label{fig:shape_front}
\end{figure}
\begin{definition}[retracting front]
Let $a$ denote a positive quantity. A solution $(x,t)\mapsto A(x,t)$ of equation \cref{nlche} is called a \emph{retracting front of amplitude $a$} if there exist a positive quantity $v$ (velocity) and a smooth function $B:\rr\rightarrow\cc$ such that, for every $(x,t)$ in $\rr^2$, 
\[
A(x,t)=B(x-vt)e^{ia^2 t} 
\]
and such that
\[
B(\xi) \rightarrow 0
\quad{when}\quad
\xi \rightarrow -\infty
\quad{and}\quad
|B(\xi)| \rightarrow a
\quad{when}\quad
\xi \rightarrow +\infty
\,.
\]
\end{definition}
Note that in this definition only fronts travelling ``to the right'' are considered. Of course this is an arbitrary choice, since for every such retracting front, space reversibility $x\rightarrow -x$ yields the existence of a symmetric travelling front travelling ``to the left'', with similar properties. 

Relevant uniformly translating and rotating  solutions of equation \cref{nlche} actually reduce to retracting fronts, as stated by the following proposition. 
\begin{proposition}[retracting fronts are the only inhomogeneous bounded uniformly translating and rotating solutions]
\label{prop:only_ret_fr}
Every uniformly translating and rotating solution of equation \cref{nlche} with a bounded amplitude is either a spatially homogeneous solution of the form \cref{spat_hom} or a retracting front (or the image of the retracting front by the $x\rightarrow -x$ symmetry).  
\end{proposition}
This proposition will be proved along the proof of the main statement of this paper (\cref{thm} below); to be more precise \cref{prop:only_ret_fr} follows from  \cref{lem:a_mon,lem:vanishing,lem:om_pos,lem:unst_man}, proved in \cref{sec:monotonic,sec:vanishing,sec:plus_infty,sec:minus_infty}, respectively. 

Besides space reversibility and the obvious symmetries that are time translation $t\rightarrow t + c$, space translation $x\rightarrow x + c$, and rotation (phase translation) of the complex amplitude $A\rightarrow A e^{i\varphi}$, equation \cref{nlche} displays, for every $\lambda$ in $(0;+\infty)$, the following scale-invariance symmetry: 
\begin{equation}
\label{scale_inv}
(A,t,x)\longrightarrow\Bigl(\lambda A, \frac{t}{\lambda^2} , \frac{x}{\lambda} \Bigr)
\,;
\end{equation}
in particular there is no characteristic scale for the modulus $\abs{A}$ of the amplitude. Because of these symmetries, retracting fronts defined above go by three-parameter families: one parameter for space-time translation, one for rotation (phase translation) of the complex amplitude, and one for the scale invariance \cref{scale_inv}. Due to this scale invariance, we shall without loss of generality restrict ourselves in some of the next statements to retracting fronts of amplitude $1$ (those go by two-parameter families --- space-time and phase translations ---instead of three). 

We are going to distinguish two subclasses among those retracting fronts. The next two propositions are preliminary results that will ease the statement of this definition.
\begin{proposition}[the amplitude of a retracting front is nonzero and strictly increasing]
\label{prop:monotonic_amplitude}
For every retracting front of equation \cref{nlche}, the (modulus of the) amplitude is nonzero and strictly increasing on the real line. 
\end{proposition}
In other words, the function $\xi\mapsto \abs{B(\xi)}$ is strictly increasing on $\rr$ (and it vanishes nowhere). 
As for \cref{prop:only_ret_fr} above, this proposition will be proved along the proof of the main result of this paper (\cref{thm} below); to be more precise \cref{prop:monotonic_amplitude} follows from \vref{lem:a_mon}.

As a consequence, if $A(x,t)=B(x-vt)e^{it}$ is a retracting front (say of amplitude equal to $1$), the first order variation of the phase of $B$, that is the quantity:
\begin{equation}
\label{first_order_var_phase}
\imm\Bigl(\frac{B'(\xi)}{B(\xi)}\Bigl)
\end{equation}
is defined for all $\xi$ in $\rr$. 
\begin{proposition}[the phase of a retracting front is strictly increasing]
\label{prop:monotonic_phase}
For every retracting front of equation \cref{nlche}, the phase is strictly increasing on the real line. 
\end{proposition}
In other words, the first order variation \cref{first_order_var_phase} of the phase of $B(\xi)$ is positive for every real quantity $\xi$. As for \cref{prop:only_ret_fr,prop:monotonic_amplitude} above, this proposition will be proved along the proof of the main result of this paper (\cref{thm} below); to be more precise \cref{prop:monotonic_phase} follows from \vref{lem:q_pos}).

Now let us make the announced distinction. 
\begin{definition}
A retracting front $(x,t)\mapsto B(x-vt) e^{it}$ (of amplitude $1$) is said to be:
\begin{enumerate}
\item \emph{steep} if:
\begin{itemize}
\item the rate at which $|B(\xi)|$ approaches $1$ when $\xi$ approaches $+\infty$ is exponential, namely:
\[
1-|B(\xi)| = o_{\xi\rightarrow+\infty}\bigl(e^{-v\xi}\bigr)
\,;
\]
\item and the variation of the phase when $\xi$ approaches $+\infty$ is finite, in other words:
\[
\int_0^{+\infty} \imm\Bigl(\frac{B'(\xi)}{B(\xi)}\Bigl) \, d\xi < +\infty
\,;
\]
\end{itemize}
\item \emph{gradual} if:
\begin{itemize}
\item the rate at which $|B(\xi)|$ approaches $1$ when $\xi$ approaches $+\infty$ is polynomial, namely:
\[
1-|B(\xi)| = \ooo_{\xi\rightarrow+\infty}\Bigl(\frac{1}{\xi}\Bigr)
\,;
\]
\item and the variation of the phase when $\xi$ approaches $+\infty$ is infinite, in other words:
\[
\int_0^{+\infty} \imm\Bigl(\frac{B'(\xi)}{B(\xi)}\Bigl) \, d\xi = +\infty
\,.
\]
\end{itemize}
\end{enumerate}
\end{definition}
This definition presents a striking analogy with front propagation into unstable states (say in the ``pulled'' case), where fronts often go by one-parameter families containing one one hand a single ``pulled'' front propagating at the ``linear spreading velocity'' and attracting localized initial conditions, and on the other hand for every larger velocity a front propagating at this larger ``leading edge dominated'' velocity \cite[p. 70]{VanSaarloos_frontPropagationUnstableStates_2003}, \cite{KolmogorovPetrovskii_diffusionEquation_1937}. Thus the ``steep'' versus ``gradual'' cases defined above may be related to those ``pulled-linear spreading'' versus ``leading edge dominated'' types of unstable fronts, respectively. Here the situation is quite different however: the invaded equilibrium is neutral, not unstable, and it is the phase gradient generated by the nonlinearity combined with the diffusion term that moves the front in favour of the zero-amplitude state. 

Our main result, illustrated by \cref{fig:partition_velocity_line}, is the following. 
\begin{theorem}[one-parameter family of retracting fronts]
\label{thm}
There exists a partition of the real line $\rr$ into three nonempty subsets:
\[
\rr = \vnone\sqcup\vSteep\sqcup\vGradual
\]
such that the following assertions hold. 
\begin{enumerate}
\item For every velocity $v$ in $\vSteep\sqcup\vGradual$, there exists a unique (up to space-time translation and rotation of complex amplitude) retracting front of amplitude $1$ and velocity $v$ for equation \cref{nlche}. This retracting front is:
\begin{itemize}
\item steep if $v$ is in $\vSteep$;
\item gradual if $v$ is in $\vGradual$.
\end{itemize}
\item For every $v$ in $\vnone$, no retracting front of amplitude $1$ and velocity $v$ exists for equation \cref{nlche}. 
\item The sets $\vnone$ and $\vGradual$ are open subsets of $\rr$, and
\[
(-\infty,0]\subset \vnone
\quad\mbox{and}\quad
[2,+\infty)\subset \vGradual
\,.
\]
\end{enumerate}
\end{theorem}
Thus $\vSteep$ is a closed non-empty subset of $(0,2)$.
\begin{figure}[!htbp]
	\centering
    \includegraphics[width=0.8\textwidth]{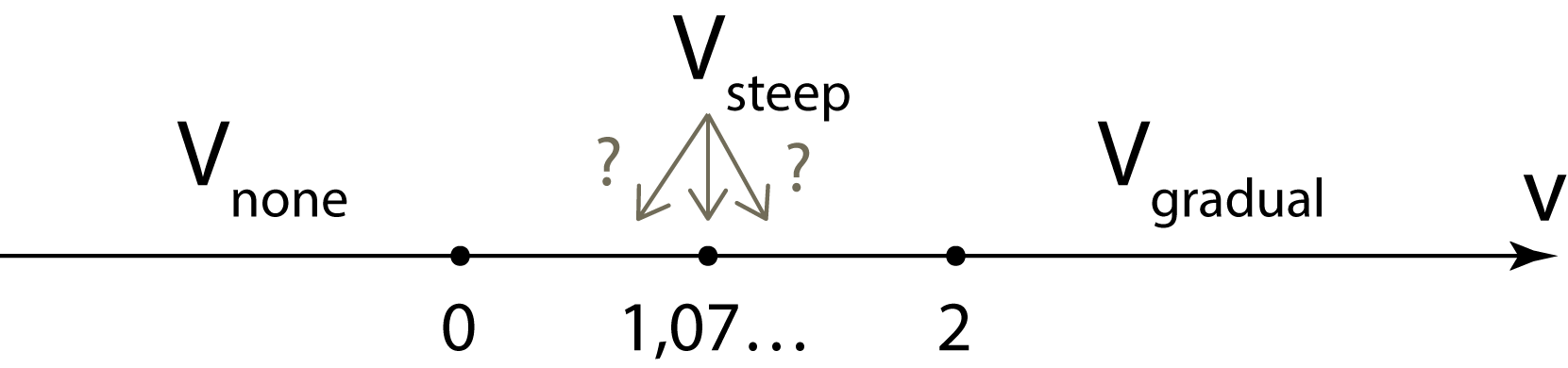}
    \caption{Partition of the velocity line (illustration of \cref{thm}).}
    \label{fig:partition_velocity_line}
\end{figure}

There is numerical evidence that the set $\vSteep$ is reduced to a single point (the corresponding value approximately equals $1.07$, see \cref{fig:partition_velocity_line}) but unfortunately we were unable to provide a proof of that. If this could be done, it would prove that the unique ``steep'' retracting front is the slowest among all retracting fronts of a given amplitude, reinforcing the analogy with front propagation into unstable states. 

There is also numerical evidence that, for every initial condition connecting the trivial solution $A\equiv0$ to a spatially homogeneous solution \cref{spat_hom} of a given amplitude $a$, the solution converges towards a retracting front of amplitude $a$ (this was already reported by Coullet and Kramer in \cite{CoulletKramer_retractingFronts_2004}). But this open question is by far beyond the scope of this paper. 
\section{Notation}
\label{sec:notation}
Let $v$ and $\omega$ be two real quantities, and $\xi\mapsto B(\xi)$ be a smooth complex-valued function of the real variable $\xi$. The function
\[
(x,t)\mapsto B(x-vt) e^{i\omega t}
\]
is a solution of equation \cref{nlche} if $B$ is a solution of the following equation:
\begin{equation}
\label{ode_B}
-vB'+i\omega B=i\abs{B}^2 B+B''
\,.
\end{equation}
Let us assume that the function $\xi\mapsto B(\xi)$ never vanishes and, proceeding as van Saarloos and Hohenberg in \cite{vanSaarloosHohenberg_frontsPulses_1992}, let us write $B(\xi)$ in polar coordinates and take the following notation:
\begin{equation}
\label{not_akq}
B(\xi)=a(\xi) e^{i\theta(\xi)}
\quad\mbox{and}\quad
q(\xi)=\theta'(\xi)
\quad\mbox{and}\quad 
\kappa(\xi)=\frac{a'(\xi)}{a(\xi)}
\quad\mbox{and}\quad 
z(\xi)= \kappa(\xi)+iq(\xi)
\,.
\end{equation}
The quantities $v$ and $\omega$ will be respectively called \emph{velocity} and \emph{frequency}. 
The quantities $a(\xi)$ (the modulus of the complex amplitude $A$), $\theta$, and $q$ will be respectively called \emph{amplitude}, \emph{phase}, and \emph{derivative of the phase}. 

According to this notation,
\[
B' = z B
\quad\mbox{and}\quad
B'' = (z' + z^2) B
\]
and equation \cref{ode_B} is equivalent to the following differential system in $\rr\times\cc$:
\begin{equation}
\label{syst_az}
\left\{
\begin{aligned}
da/d\xi &= (\ree z) a \\
dz/d\xi &= i(\omega-a^2)-vz-z^2
\end{aligned}
\right.
\end{equation}
or to the following differential system in $\rr^3$:
\begin{equation}
\label{syst_akq}
\left\{
\begin{aligned}
da/d\xi &= \kappa a \\
d\kappa/d\xi &= -v\kappa - \kappa^2 + q^2 \\
dq/d\xi &= \omega-a^2-vq-2q\kappa
\end{aligned}
\right.
\end{equation}
Besides, we see from \cref{syst_akq} that $\xi\mapsto a(\xi)$ satisfies the second order non-autonomous differential equation:
\begin{equation}
\label{2nd_order_a}
a''=q^2 a-va'
\,. 
\end{equation}
A mechanical interpretation of equation \cref{2nd_order_a} (the fact that the ``force'' term $q^2 a$ is repulsive) already supports the idea that bounded solutions of system \cref{syst_akq} might not be numerous (this will be formalized in \cref{lem:a_mon} below).
\section{Sketch of the proof and organization of the paper}
The (elementary) proof of \cref{thm} is summarized on \cref{fig:phase_space}. It follows from a dynamical study of system \cref{syst_akq}, in particular of all bounded solutions. We will show that the only relevant solutions lie on the unstable manifold of the hyperbolic saddle-focus $z_+(v)$ in the invariant plane $\{a=0\}$, and that this unstable manifold may either go to infinity, or converge towards the equilibrium corresponding to spatially homogeneous solutions \cref{spat_hom}, and that this convergence may occur either through the two-dimensional stable manifold of this equilibrium, or through its one-dimensional center manifold. The partition of $\rr$ into three subsets stated in \cref{thm} follows. 
\begin{figure}[!htbp]
	\centering
    \includegraphics[width=\textwidth]{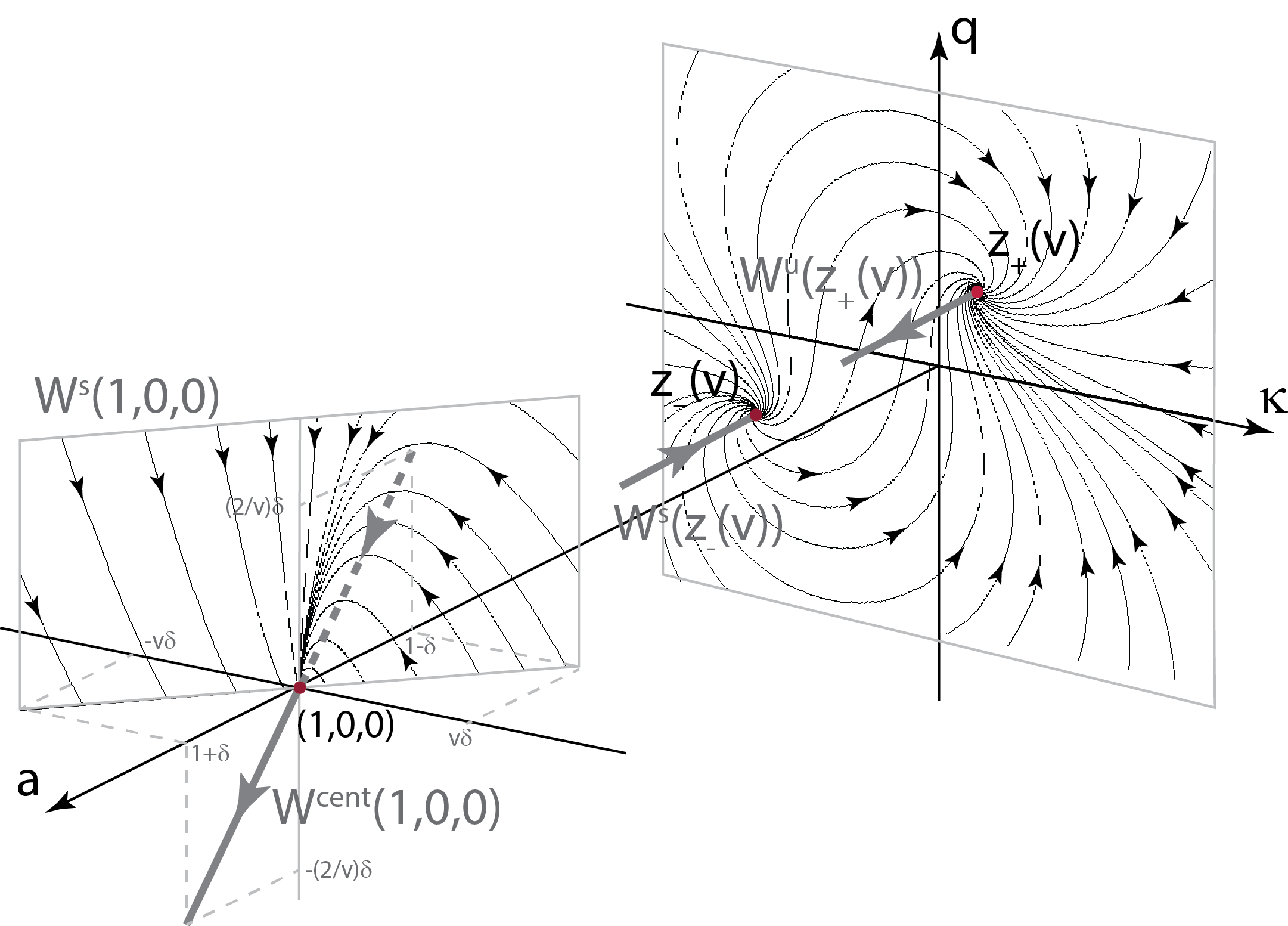}
    \caption{Phase space of equation \cref{syst_akq} in the case $\omega=1$. Equilibrium $(1,0,0)$ corresponds to spatially homogeneous time-periodic solutions of amplitude $1$ for equation \cref{nlche}. This equilibrium has a two-dimensional stable manifold and a one-dimensional center manifold. The center-manifold is tangent to the direction of eigenvector $(1,0,-2/v)$, and the quantity $a$ is increasing along solutions close to $(1,0,0)$ on this center manifold. Trajectories in the stable manifold of $(1,0,0)$ approach this point tangentially to the (vertical) eigenvector $(0,0,1)$. Retracting fronts correspond to homoclinic connections between $z_+(v)$ and $(1,0,0)$. Gradual retracting fronts approach $(1,0,0)$ though its center manifold, while for steep retracting fronts it is through the stable manifold. Trajectories in the $\{a=0\}$ plane and in the stable space of $(1,0,0)$ were drawn in the case $v=1$. Note that some scales are incorrect (for instance $\abs{z_-(v)}$ is larger than $1$ whereas it looks the opposite on the figure).}
    \label{fig:phase_space}
\end{figure}

Symmetries of system \cref{syst_akq} are stated in \cref{sec:symmetries}. In \cref{sec:monotonic} we use equation \cref{2nd_order_a} (and its mechanical interpretation) to show that relevant solutions must have a monotonic amplitude, with a sense of monotonicity opposite to that of the velocity $v$ (whih plays the role of a damping coefficient in this equation). The very same idea is used in \cref{sec:vanishing} to show that solutions of the second order differential equation \cref{ode_B} that vanish at some point are unbounded. In \cref{sec:plus_infty} we show that relevant solutions exist only if $\omega$ is positive, and that they must approach the equilibrium $(\sqrt{\omega},0,0)$ (corresponding to spatially homogeneous solutions \cref{spat_hom} of amplitude $\sqrt{\omega}$) at plus infinity.
In \cref{sec:neighb_equil} we observe that this equilibrium has a two-dimensional stable manifold and a one-dimensional center manifold, and we give a rough description of this center manifold. Dynamics in the invariant plane $\{a=0\}$ is studied (recalled) in \cref{sec:a_equals_0}. In \cref{sec:minus_infty} we prove that the sole relevant solutions lie in the unstable manifold of equilibrium $z_+(v)$. Finally, the dynamical behaviour on this unstable manifold and the final splitting argument is given in \cref{sec:unst_man}. 
\section{Symmetries}
\label{sec:symmetries}
The three following symmetries hold for system \cref{syst_akq} and the parameters $v$ and $\omega$:
\begin{align}
a &\longrightarrow -a \label{a_-a_sym}\\
(\xi,\kappa,q,v) &\longrightarrow (-\xi,-\kappa,-q,-v) \label{rev_sym}\\
(\xi,a,\kappa,q,v,\omega) &\longrightarrow (\xi/\lambda, \lambda a, \lambda \kappa,q,\lambda v,\lambda^2 \omega)\quad\mbox{for every positive quantity}\quad \lambda\,.\label{scale_sym}
\end{align}
They can be viewed, respectively, as consequences of the following properties of initial equation \cref{nlche}:
\begin{itemize}
\item rotation invariance: $A\mapsto A e^{i\pi}$,
\item space reversibility: $x\leftrightarrow -x$,
\item scale invariance: $(A,t,x)\leftrightarrow(\lambda A, t/\lambda^2 , x/\lambda)$.
\end{itemize}
According to $a\rightarrow-a$-symmetry \cref{a_-a_sym}, it is sufficient to study system \cref{syst_akq} on half-space $\bigl\{(a,\kappa,q)\in\rr^3:a\ge0\bigr\}$; by the way the plane $\bigl\{(a,\kappa,q)\in\rr^3:a=0\bigr\}$ is invariant, as are the two complementary open half-spaces. Obviously, for every solution to this system (defined on a single interval), the function $\xi\mapsto a(\xi)$ either identically vanishes or never vanishes. 

According to space reversibility \cref{rev_sym}, it is sufficient to study system \cref{syst_akq} either for $v$ nonnegative or nonpositive. 

Outside of the invariant plane $\bigl\{(a,\kappa,q)\in\rr^3:a=0\bigr\}$, system \cref{syst_akq} has no equilibrium if $\omega$ is nonpositive, and two symmetric equilibria at:
\[
(a,\kappa,q)=(\pm\sqrt{\omega},0,0)
\]
if $\omega$ is positive. These two equilibria correspond to spatially homogeneous solutions \cref{spat_hom} of equation \cref{nlche}. 

According to scale invariance \cref{scale_sym}, frequency $\omega$ can be normalized to $1$, $0$, or $-1$. 
In view of the fact that the nonlinear term in initial equation \cref{nlche} ``pushes in the trigonometric sense'', nonpositive values of $\omega$ are expected to be irrelevant (this will be proved below). 
\section{Monotonicity of relevant solutions and sign of the velocity}
\label{sec:monotonic}
In view of the expression of $a''$ in equation \cref{2nd_order_a}, since the force term in this equation is repulsive, we expect that, if at some point the derivative $a'$ does not have the same sign as the velocity $v$ (which plays in \cref{2nd_order_a} the role of a dissipation coefficient), then the amplitude $a$ must diverge either in the future on in the past. The following lemma formalizes this heuristics. 
\begin{lemma}[the sign of $a'(\cdot)$ is always equal to that of $v$, otherwise amplitude is unbounded]
\label{lem:a_mon}
Let $\xi\mapsto\bigl(a(\xi),\kappa(\xi),q(\xi)\bigr)$ denote a nonconstant solution of system \cref{syst_akq}, defined on a maximal existence interval $(\xi_{\min},\xi_{\max})$ --- with $-\infty\le \xi_{\min}<\xi_{\max}\le+\infty$, and taking values in half-space $\{(a,\kappa,q)\in\rr^3:a>0\}$. 
Then, 
\begin{enumerate}
\item If $v\le0$ and there exists $\xi_0$ in $(\xi_{\min},\xi_{\max})$ such that $a'(\xi_0)\ge 0$, 
then $a(\xi)$ approaches $+\infty$ when $\xi$ approaches $\xi_{max}$. 
\item If $v\ge0$ and there exists $\xi_0$ in $(\xi_{\min},\xi_{\max})$ such that $a'(\xi_0)\le 0$, 
then $a(\xi)$ approaches $+\infty$ when $\xi$ approaches $\xi_{min}$. 
\end{enumerate}
\end{lemma}
\begin{proof}
According to the space reversibility symmetry \cref{rev_sym}, assertions 1 and 2 of this lemma are equivalent. Let us prove assertion 1. Thus we assume that $v$ is nonpositive and that there exists $\xi_0$ in $(\xi_{\min},\xi_{\max})$ such that $a'(\xi_0)$ is nonnegative. 

The main step is to find a value of $\xi$ (equal to or slightly larger than $\xi_0$) such that $a'(\xi)$ is positive. According to \cref{2nd_order_a}, one of the three following cases occurs:
\begin{enumerate}
\item $a'(\xi_0)>0$,
\item $a'(\xi_0)=0$ and $q(\xi_0)\not=0$, thus $a''(\xi_0)>0$,
\item $a'(\xi_0)=0$ and $q(\xi_0)=0$ equals zero, thus $a''(\xi_0)=0$; in this case $q'(\xi_0)\not=0$, or else the solution would be constant.  
\end{enumerate}
According to \cref{2nd_order_a}, 
\[
a'''=2qq'a + q^2 a' - v a'' 
\quad\mbox{and}\quad
a''''=2q'^2 a + 2qq''a +4qq'a' + q^2 a'' - v a'''
\,,
\]
thus in the third of the three cases above, we get:  
\[
a'''(\xi_0)=0
\quad\mbox{and}\quad
a''''(\xi_0)=2 q'^2(\xi_0) a(\xi_0)>0
\,.
\]
In short, $a'$ must be positive at some place, either at $\xi_0$ or immediately after. As a consequence, there exists $\xi_1$ in $[\xi_0,\xi_{\max})$ such that $a'(\xi_1)>0$. Then, according to \cref{2nd_order_a}, 
\begin{equation}
\label{low_bd_aprime}
a'(\xi)\ge a'(\xi_1)
\quad\mbox{for all}\quad
\xi\mbox{ in } [\xi_1,\xi_{\max})
\,.
\end{equation}

Let us proceed by contradiction and assume that $a(\xi)$ does not approach $+\infty$ when $\xi$ approaches $\xi_{\max}$. Then, according to \cref{low_bd_aprime}, we must have $\xi_{\max}<+\infty$. 

But on the other hand, equation \cref{syst_az} on $dz/d\xi$ yields, for every $\xi$ in $(\xi_{\min},\xi_{\max})$:
\begin{equation}
\label{d_abs_z}
\frac{d\abs{z}}{d\xi}\le -(v + \ree z)\abs{z} + \abs{\omega-a^2}
\end{equation}
and since $\ree z(\xi)=a'(\xi)/a(\xi)$ is positive for all $\xi$ in $[\xi_1,\xi_{\max})$, this shows that blow-up at $\xi_{\max}$ cannot occur, a contradiction. The lemma is proved. 
\end{proof}
According to \cref{lem:a_mon}, solutions of system \cref{syst_akq} corresponding to retracting fronts (for which the amplitude approaches zero when $\xi$ approaches $-\infty$ and a positive value when $\xi$ approaches $+\infty$) can exist only for $v$ positive, and have a strictly increasing amplitude (this proves \vref{prop:monotonic_amplitude}). 
\section{Non relevance of solutions with vanishing amplitude}
\label{sec:vanishing}
The aim of this short \namecref{sec:vanishing} is to get rid of solutions of equation \cref{ode_B} that vanish at some point. We do this now since the argument is exactly the same as in the proof of \cref{lem:a_mon} above. 
\begin{lemma}[amplitude does not vanish, otherwise it is unbounded]
\label{lem:vanishing}
Let $\xi\mapsto B(\xi)$ denote a solution of equation \cref{ode_B}, defined on a maximal existence interval $(\xi_{\min},\xi_{\max})$ --- with $-\infty\le \xi_{\min}<\xi_{\max}\le+\infty$, and such that there exists $\xi_0$ in $(\xi_{\min},\xi_{\max})$ such that $B(\xi_0)=0$. Assume furthermore that this solution is not identically zero. Then, 
\begin{enumerate}
\item if $v$ is nonpositive, then $\abs{B(\xi)}$ approaches $+\infty$ when $\xi$ approaches $\xi_{\max}$;
\item if $v$ is nonnegative, then $\abs{B(\xi)}$ approaches $+\infty$ when $\xi$ approaches $\xi_{\min}$;
\end{enumerate}
\end{lemma}
\begin{proof}
According to the space reversibility symmetry \cref{rev_sym}, assertions 1 and 2 of this lemma are equivalent. Let us prove assertion 1. Thus let us assume that $v$ is nonpositive. Since $B(\xi_0)=0$ and $B(\cdot)$ is not identically zero, $B'(\xi_0)$ is not equal to zero. 

Let $\tilde\xi_{\max}$ denote the supremum of the (nonempty) set:
\[
\{\tilde\xi\in\rr:\tilde\xi>\xi_0 \mbox{ and } B(\xi)\not=0 \mbox{ for all } \xi \mbox{ in } (\xi_0,\tilde\xi)\}
\,.
\]
Since $B(\cdot)$ does not vanish on $(\xi_0,\tilde\xi_{\max})$, definition \cref{not_akq} of $(a,\kappa,q)$ is not singular on this interval (it is definitely singular at $\xi_0$), thus this definition provides a solution:
\[
(\xi_0,\tilde\xi_{\max})\rightarrow\rr_+^*\times\rr^2,
\quad
\xi\mapsto\bigl(a(\xi),\kappa(\xi),q(\xi)\bigr)
\]
of system \cref{syst_akq}. Since $a(\xi)$ approaches $0$ when $\xi$ approaches $\xi_0$, there must exists $\xi_1$ larger than $\xi_0$ (close to $\xi_0$) such that $a'(\xi)$ is positive. It thus follows from \cref{lem:a_mon} that $a(\xi)$ approaches $+\infty$ when $\xi$ approaches $\tilde\xi_{\max}$, and by the way that the function $\xi\mapsto a(\xi)$ is strictly increasing on $(\xi_0,\tilde\xi_{\max}$ and that $\tilde\xi_{\max}$ equals $\xi_{\max}$. \Cref{lem:vanishing} is proved. 
\end{proof}
These non relevant solutions will be met in \cref{sec:minus_infty} below.
\section{Asymptotics at plus infinity}
\label{sec:plus_infty}
\begin{lemma}[amplitude converges at plus infinity, otherwise it is unbounded]
\label{lem:om_pos}
Let $\xi\mapsto\bigl(a(\xi),\kappa(\xi),q(\xi)\bigr)$ denote a nonconstant solution of system \cref{syst_akq}, defined on a maximal existence interval $(\xi_{\min},\xi_{\max})$ --- with $-\infty\le \xi_{\min}<\xi_{\max}\le+\infty$, and taking values in half-space $\{(a,\kappa,q)\in\rr^3:a>0\}$. Assume that:
\begin{itemize}
\item $v$ is positive,
\item the amplitude $\xi\mapsto a(\xi)$ is strictly increasing on $(\xi_{\min},\xi_{\max})$,
\item the amplitude does not approach $+\infty$ when $\xi$ approaches $\xi_{\max}$. 
\end{itemize}
Then the following conclusions hold: 
\begin{enumerate}
\item the upper bound $\xi_{\max}$ of the interval of existence of the solution is equal to $+\infty$,
\item the frequency $\omega$ is positive,
\item the amplitude $a(\xi)$ approaches $\sqrt{\omega}$ when $\xi$ approaches $+\infty$. 
\end{enumerate}
\end{lemma}
\begin{proof}
Since $\ree z(\xi)=a'(\xi)/a(\xi)$ is positive for all $\xi$ in $(\xi_{\min},\xi_{\max})$, inequality \cref{d_abs_z} on $d\abs{z}/d\xi$ shows that $\xi_{\max}$ must be equal to $+\infty$; assertion 1 is proved. 

Since $v$ is positive, inequality \cref{d_abs_z} even shows that $\abs{z}(\xi)$ remains bounded when $\xi$ approaches $+\infty$. In other words $\kappa(\xi)$ and $q(\xi)$ remain bounded when $\xi$ approaches $+\infty$. Expressions of $da/d\xi$ in \cref{syst_akq} and of $a''=d^2a/d\xi^2$ in \cref{2nd_order_a} show that $a'(\xi)$ and $a''(\xi)$ remain in turn bounded when $\xi$ approaches $+\infty$. 

On the other hand, according to its boundedness and monotonicity, the amplitude $a(\xi)$ must approach a finite limit when $\xi$ approaches $+\infty$, and the boundedness of $a''(\xi)$ shows that $a'(\xi)$ must go to zero when $\xi$ approaches $+\infty$ (thus the same is true for $\kappa(\xi)$). 

Again, the quantities $q'(\xi)$ and thus $a'''(\xi)$ are bounded when $\xi$ approaches $+\infty$. This shows that $a''(\xi)$ approaches zero when $\xi$ approaches $+\infty$ and, according to expression \cref{2nd_order_a} of $a''(\xi)$, this shows that $q(\xi)$ approaches zero when $\xi$ approaches $+\infty$. 

According to the expression of $q'(\xi)$ in system \cref{syst_akq}, this shows that $\omega$ is positive and that the limit of $a(\xi)$ when $\xi$ approaches $+\infty$ must be equal to $\sqrt{\omega}$. Assertions 2 and 3 are proved. 
\end{proof}
This lemma shows that no relevant solution can be found for system \cref{syst_akq} if the frequency $\omega$ is nonpositive. From now on, we thus assume that $\omega$ is positive, and, according to the scale-invariance symmetry \cref{scale_inv}, we assume without loss of generality that:
\[
\omega = 1
\,.
\]
Systems \cref{syst_akq} and \cref{syst_az} and equation \cref{ode_B} thus become, respectively:
\begin{equation}
\label{syst_akq_om_1}
\left\{
\begin{aligned}
da/d\xi &= \kappa a \\
d\kappa/d\xi &= -v\kappa +q^2-\kappa^2 \\
dq/d\xi &= 1-a^2-vq-2q\kappa
\end{aligned}
\right.
\end{equation}
and
\begin{equation}
\label{syst_az_om_1}
\left\{
\begin{aligned}
da/d\xi &= (\ree z) a \\
dz/d\xi &= i(1-a^2)-vz-z^2
\end{aligned}
\right.
\end{equation}
and
\begin{equation}
\label{ode_B_om_1}
-vB'+i B=i\abs{B}^2 B+B''
\,.
\end{equation}

Recall that the velocity $v$ is assumed to be positive. 

According to \cref{lem:om_pos}, every relevant solution of system \cref{syst_akq_om_1} in the half-space 
\[
\{(a,\kappa,q)\in\rr^3:a>0\}
\]
must converge, when $\xi$ approaches $+\infty$, to the equilibrium $(1,0,0)$ corresponding to spatially homogeneous time-periodic solutions of amplitude $1$. Let us now look at the dynamics close to this equilibrium. 
\section{Dynamics in a neighbourhood of the finite amplitude equilibrium}
\label{sec:neighb_equil}
We assume that the velocity $v$ is positive. 
The matrix of the linearization of system \cref{syst_akq_om_1} at $(1,0,0)$ reads:
\[
\begin{pmatrix}
0 & 1 & 0 \\
0 & -v & 0 \\
-2 & 0 & -v
\end{pmatrix}
\quad\mbox{in the canonical basis of $\rr^3$}
\]
and
\begin{equation}
\label{lin}
\begin{pmatrix}
-v&2&0\cr 0&-v&0\cr 0&0&0
\end{pmatrix}
\quad\mbox{in the basis:}\quad 
\begin{pmatrix}
0 \\ 0 \\ 1
\end{pmatrix},
\begin{pmatrix}
-1 \\ v \\ 0
\end{pmatrix},
\begin{pmatrix}
v \\ 0 \\ -2
\end{pmatrix}
\end{equation}
(see \vref{fig:phase_space}). This equilibrium thus admits:
\begin{itemize}
\item a two-dimensional stable manifold tangent at $(1,0,0)$ to the plane orthogonal to the vector $(v,1,0)$,
\item a one-dimensional center manifold tangent at $(1,0,0)$ to the direction of $(v,0,-2)$.
\end{itemize}

The (non-unique) center manifold, viewed as the graph of a function: $a\mapsto\bigl(\kappa(a),q(a)\bigr)$ for $a\simeq1$, admits the following second-order approximation (where $\kappa_2$ and $q_2$ are real constants):
\begin{align}
\kappa(a) &= \kappa_2 (1-a)^2 + \ooo_{a\rightarrow 1}\bigl((1-a)^3\bigr) \label{kappa_cent}\\
q(a) &= \frac{2}{v}(1-a) + q_2 (1-a)^2 + \ooo_{a\rightarrow 1}\bigl((1-a)^3\bigr) \nonumber
\end{align}
Replacing this ansatz into system \cref{syst_akq_om_1}, we obtain:
\[
\kappa_2 = \frac{4}{v^3}
\quad\mbox{and}\quad
q_2 = \frac{8}{v^5} - \frac{1}{v}
\,.
\]
The fact that $\kappa_2$ is positive shows that, close to $(1,0,0)$, the $a$-component of solutions on the center manifold is increasing with $\xi$. As a consequence, in a neighbourhood of $(1,0,0)$:
\begin{itemize}
\item the center manifold is unique on one side, where the dynamics drives us away from $(1,0,0)$, namely for $a>1$;
\item on the other side, where the dynamics brings us closer to $(1,0,0)$, namely for $a<1$, every solution is on a center manifold. 
\end{itemize}

Let $\xi\mapsto\bigl(a(\xi),\kappa(\xi),q(\xi)\bigr)$ denote a solution of system \cref{syst_akq_om_1} defined on a maximal time interval $(\xi_{\min},+\infty)$ --- with $-\infty\le\xi_{\min}<0$. 
\begin{enumerate}
\item If this solution belongs to the stable manifold of $(1,0,0)$, then it approaches this point tangentially to the vector $(0,0,1)$ when $\xi$ approaches $+\infty$. As a consequence, 
\begin{equation}
\label{approach_stab}
\norm{\bigl(a(\xi)-1,\kappa(\xi),q(\xi)\bigr)} = \ooo_{\xi\rightarrow+\infty}\bigl(e^{-v\xi}\bigr)
\quad\mbox{and}\quad
1-a(\xi) = o_{\xi\rightarrow+\infty}\bigl(e^{-v\xi}\bigr)
\,,
\end{equation}
and
\begin{equation}
\label{tot_phase_stab}
\int_0^{+\infty} q(\xi)\,d\xi < +\infty
\,.
\end{equation}
\item If this solution belongs to the center manifold of $(1,0,0)$, then it approaches this point tangentially to the vector $(-v,0,2)$, and from expression \cref{kappa_cent} it follows that:
\begin{equation}
\label{approach_cent}
1-a(\xi) = \frac{v^3}{4\xi} + o_{\xi\rightarrow+\infty}\Bigl(\frac{1}{\xi}\Bigr)
\quad\mbox{and}\quad
q(\xi) = \frac{v^2}{2\xi} + o_{\xi\rightarrow+\infty}\Bigl(\frac{1}{\xi}\Bigr)
\,,
\end{equation}
thus
\begin{equation}
\label{tot_phase_center}
\int_0^{+\infty} q(\xi)\,d\xi = +\infty
\,.
\end{equation}
\end{enumerate}
\section{Dynamics in the zero amplitude invariant plane}
\label{sec:a_equals_0}
The dynamics of system \cref{syst_az_om_1} in the invariant plane $\bigl\{(a,z)\in\rr\times\cc:a=0\bigr\}$ is governed by the differential equation: 
\begin{equation}
\label{equ_z_a_eq_0}
dz/d\xi= i-vz-z^2
\quad\mbox{where}\quad
z = \kappa + i q
\end{equation}
(see \vref{fig:phase_space}).
Let $\alpha+i\beta$ denote the square root of $v^2+4i$ with positive real part, and let
\[
z_-(v) = \frac{-v -\alpha-i\beta}{2}
\quad\mbox{and}\quad
z_+(v) = \frac{-v +\alpha+i\beta}{2}
\]
denote the two equilibria of equation \cref{equ_z_a_eq_0}. From $\alpha^2-\beta^2 = v^2$ and $\alpha\beta=2$, we get:
\[
\ree z_-(v) < 0 < \ree z_+(v) 
\quad\mbox{and}\quad
\imm z_-(v) < 0 < \imm z_+(v) 
\,,
\]
and, if $v>0$, 
\begin{equation}
\label{abs_z_plus}
\abs{z_+(v)}<1<\abs{z_-(v)}
\,.
\end{equation}
The change of variables:
\[
Z=\frac{z-z_+(v)}{z-z_-(v)}
\]
transforms equation \cref{equ_z_a_eq_0} into:
\[
dZ/d\xi=-\bigl((z_+(v)-z_-(v)\bigr)\,Z
\,.
\]
Thus, with respect to the dynamics in the plane $\{a=0\}$, $z_-(v)$ and $z_+(v)$ are a repulsive focus and an attractive focus, respectively; and with respect to the dynamics of system \cref{syst_akq} in $\rr^3$, both are hyperbolic saddle-focus (transversely to the plane $\{a=0\}$ the equilibrium $z_-(v)$ is attractive whereas $z_+(v)$ is repulsive). All this is shown on \vref{fig:phase_space}.

In the following we will make for convenience the following notation abuse: we will denote by $z_+(v)$ the point $\bigl(0,\ree z_+(v),\imm z_+(v)\bigr)$ in the phase spase $\rr^3$ of system \cref{syst_akq_om_1} --- or the point $\bigl(0, z_+(v)\bigr)$ in the phase space $\rr\times\cc$ of system \cref{syst_az_om_1}. 
\section{Asymptotics at minus infinity}
\label{sec:minus_infty}
The assumptions of \cref{lem:unst_man} below are almost identical to those of the previous \cref{lem:a_mon,lem:om_pos}. But this time, we shall take care of the asymptotics of the solution when $\xi$ goes to $-\infty$. Because of the existence of solutions $\xi\mapsto B(\xi)$ of equation \cref{ode_B} that vanish at some point (see \cref{sec:vanishing}), these asymptotics are slightly more involved than at $+\infty$. 
\begin{lemma}[only solutions in the unstable manifold of $z_+(v)$ have a bounded amplitude]
\label{lem:unst_man}
Let $\xi\mapsto\bigl(a(\xi),\kappa(\xi),q(\xi)\bigr)$ denote a nonconstant solution of system \cref{syst_akq_om_1}, defined on a maximal existence interval $(\xi_{\min},\xi_{\max})$ --- with $-\infty\le \xi_{\min}<\xi_{\max}\le+\infty$ --- and such that the amplitude $\xi\mapsto a(\xi)$ is positive. Assume that $v$ is positive and that the amplitude $\xi\mapsto a(\xi)$ is bounded, that is:
\[
\sup_{\xi\in(\xi_{\min},\xi_{\max})} \abs{a(\xi)} < +\infty
\,.
\]
Then the amplitude $a(\xi)$ approaches zero when $\xi$ approaches $\xi_{\min}$ from the right, and one among the two following cases occurs.
\begin{enumerate}
\item The lower bound $\xi_{\min}$ is finite. In this case there exists a solution $\xi\mapsto B(\xi)$ of equation \cref{ode_B_om_1}, defined on an interval containing $[\xi_{\min},\xi_{\max})$, vanishing at $\xi_{\min}$, and corresponding to $\xi\mapsto\bigl(a(\xi),\kappa(\xi),q(\xi)\bigr)$ --- through the correspondence \cref{not_akq} --- on $(\xi_{\min},\xi_{\max})$.
\item The lower bound $\xi_{\min}$ equals $-\infty$. In this case the solution $\xi\mapsto\bigl(a(\xi),\kappa(\xi),q(\xi)\bigr)$ belongs to the unstable manifold of the equilibrium $z_+(v)$, in other words:
\[
\bigl(\kappa(\xi)+iq(\xi)\bigr) \rightarrow z_+(v)
\quad\mbox{when}\quad \xi\rightarrow -\infty
\,.
\]
\end{enumerate}
\end{lemma}
According to the conclusions of \cref{lem:vanishing}, if $\xi_{\min}$ is finite (item 1 above) the solution under consideration is not relevant.  
\begin{proof}
According to \cref{lem:a_mon}, the amplitude $\xi\mapsto a(\xi)$ is strictly increasing on the interval $(\xi_{\min},\xi_{\max})$. Let $a_{-\infty}$ denote the limit of $a(\xi)$ when $\xi$ approaches $\xi_{\min}$. 

For all $\xi$ in $(\xi_{\min},\xi_{\max})$, let us write: $z(\xi)=\kappa(\xi) +i q(\xi)$. Since the frequency $\omega$ equals $1$, expression of $dz/d\xi$ in \cref{syst_az} becomes:
\begin{equation}
\label{dz_om_1}
dz/d\xi = i(1-a^2) - vz - z^2
\,.
\end{equation}
Let $\tilde\alpha+i\tilde\beta$ denote the square root of $v^2+4i(1-a_{-\infty}^2)$ with positive real part (thus both $\tilde\alpha$ and $\tilde\beta$ are positive), and let
\[
\tilde z_- = \frac{-v -\tilde\alpha-i\tilde\beta}{2}
\quad\mbox{and}\quad
\tilde z_+ = \frac{-v +\tilde\alpha+i\tilde\beta}{2}
\]
denote the two roots of the right-hand side of \cref{dz_om_1} where $a$ is replaced by $a_{-\infty}$. Since $\tilde\alpha^2 - \tilde\beta^2 = v^2$, these two roots have real parts of opposite signs, namely:
\[
\ree \tilde z_- < 0 < \ree \tilde z_+
\,.
\]
For all $\xi$ in $(\xi_{\min},\xi_{\max})$, let 
\begin{equation}
\label{z_Z_a_min_inf}
Z(\xi)=\frac{z(\xi)-\tilde z_+}{z(\xi)-\tilde z_-}
\Longleftrightarrow
z(\xi) = \frac{\tilde z_+ - \tilde z_- Z(\xi)}{1-Z(\xi)}
\end{equation}
(see \cref{fig:change_variables}).
\begin{figure}[!htbp]
	\centering
    \includegraphics[width=0.85\textwidth]{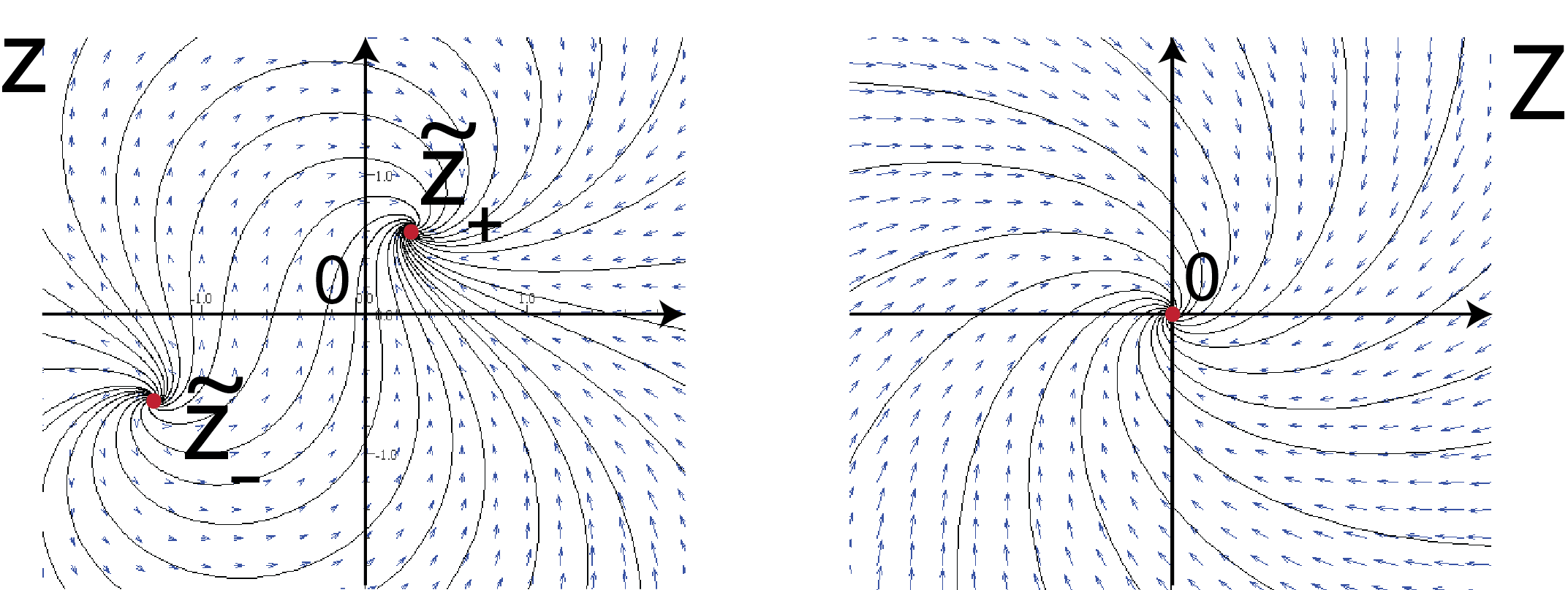}
    \caption{Change of variables \cref{z_Z_a_min_inf}).}
    \label{fig:change_variables}
\end{figure}

Observe that since the amplitude $a(\xi)$ is strictly increasing, the real part of $z(\xi)$ is positive, thus
\[
\abs{z(\xi)-\tilde z_-}\ge \abs{\ree \tilde z_-} >0
\quad\mbox{for all}\quad
\xi\mbox{ in } (\xi_{\min},\xi_{\max})
\,,
\]
so that in the previous change of variables $Z(\xi)$ is well-defined for all $\xi$ in $(\xi_{\min},\xi_{\max})$. According to \cref{dz_om_1}, 
\begin{equation}
\label{dZdxi_a_min_inf}
\begin{split}
dZ/d\xi &= -(\tilde z_+ - \tilde z_-)\,Z + i\bigl(a_{-\infty}^2-a(\xi)^2\bigr)\frac{\tilde z_+-\tilde z_-}{\bigl(z(\xi)-\tilde z_-\bigr)^2} \\
&= -(\tilde z_+ - \tilde z_-)\,Z + o(1)
\quad\mbox{as}\quad 
\xi\rightarrow\xi_{\min}
\quad\mbox{from the right.}
\end{split}
\end{equation}

First let us consider the case where $\xi_{\min}>-\infty$. Since the quantity $a(\xi)$ does not blow-up at $\xi_{\min}$, the quantity $\abs{z(\xi)}$ must approach $+\infty$ when $\xi$ approaches $\xi_{\min}$, thus the quantity $Z(\xi)$ must approach $1$ when $\xi$ approaches $\xi_{\min}$. It follows from \cref{dZdxi_a_min_inf} that
\[
Z(\xi)-1 \ \sim\ - (\tilde z_+ - \tilde z_-)\, (\xi - \xi_{\min})
\quad\mbox{as}\quad
\xi\rightarrow\xi_{\min}
\quad\mbox{from the right,}
\]
and thus from \cref{z_Z_a_min_inf} that
\[
z(\xi) \ \sim\ \frac{1}{\xi-\xi_{\min}}
\quad\mbox{as}\quad
\xi\rightarrow\xi_{\min}
\quad\mbox{from the right.}
\]
Thus, for every $\xi_0$ in $(\xi_{\min},\xi_{\max})$, the integral
\[
\int_{\xi}^{\xi_0}\ree z(\xi)\, d\xi
\]
approaches $-\infty$ as $\xi$ approaches $\xi_{\min}$ from the right, and from $da/d\xi=\bigl(\ree z(\xi)\bigr) a(\xi)$ it follows that $a_{-\infty}$ equals zero. To conclude let $\xi_0$ be a quantity in $(\xi_{\min},\xi_{\max})$ and, for every $\xi$ in $(\xi_{\min},\xi_{\max})$, let
\[
B(\xi) = a(\xi) \exp\Bigl(i\int_{\xi_0}^\xi q(\zeta)\, d\zeta\Bigr)
\,.
\]
This defines a solution of equation \cref{ode_B_om_1} defined on $(\xi_{\min},\xi_{\max})$ and such that $B(\xi)$ approaches zero as $\xi$ approaches $\xi_{\min}$ from the right. Since $\xi_{\min}$ is finite, $B(\cdot)$ must be the restriction to the interval $(\xi_{\min},\xi_{\max})$ of a solution of equation \cref{ode_B_om_1} that vanishes at $\xi_{\min}$. This finishes the proof in the case where $\xi_{\min}$ is finite. 

We are left with the case where $\xi_{\min}=-\infty$.
In this case, it follows from \cref{dZdxi_a_min_inf} that $\abs{Z(\xi)}$ approaches either zero or $+\infty$ when $\xi$ approaches $-\infty$. But since the real part of $z(\xi)$ remains positive, the second of these alternatives cannot occur. Thus $z(\xi)$ approaches $\tilde z_+$ when $\xi$ approaches $-\infty$. 

Thus $\kappa(\xi)$ approaches the positive quantity $\ree \tilde z_+$ when $\xi$ approaches $-\infty$, and from the expression of $da/d\xi$ it follows that $a_{-\infty}$ equals zero. This finishes the proof. 
\end{proof}
\cref{prop:only_ret_fr} follows from \cref{lem:a_mon,lem:vanishing,lem:om_pos,lem:unst_man}. 
\section{Unstable manifold of the zero amplitude equilibrium}
\label{sec:unst_man}
We still assume that $\omega$ equals $1$. Let $\xi\mapsto\bigl(a(\xi),\kappa(\xi),q(\xi)\bigr)$ denote a solution of system \cref{syst_akq_om_1} taking its values in the part of the unstable manifold of $z_+(v)$ situated in the open half-space $\{(a,\kappa,q)\in\rr^3 : a>0\}$, and defined on a maximal interval $(-\infty,\xi_{\max})$, with $\xi_{\max}\le+\infty$. 

According to \cref{lem:a_mon}, 
\begin{itemize}
\item the amplitude $\xi\mapsto a(\xi)$ is strictly increasing on $(-\infty,\xi_{\max})$,
\item and if $v$ is nonpositive, then $a(\xi)$ approaches $+\infty$ when $\xi$ approaches $\xi_{\max}$.
\end{itemize}
According to \cref{lem:om_pos}, if $v$ is positive, then:
\begin{itemize}
\item either $a(\xi)$ approaches $+\infty$ when $\xi$ approaches $\xi_{\max}$, 
\item or $\xi_{\max}$ equals $+\infty$ and the amplitude $a(\xi)$ approaches $1$ when $\xi$ approaches $\xi_{\max}$.
\end{itemize}
As a consequence, for every $v$ in $\rr$, there exists a unique solution 
\[
\xi\mapsto\bigl(a_v(\xi),\kappa_v(\xi),q_v(\xi)\bigr)
\]
of system \cref{syst_akq_om_1} taking its values in the part of the unstable manifold of $z_+(v)$ situated in the open half-space $\{(a,\kappa,q)\in\rr^3 : a>0\}$, defined on a maximal interval $\bigl(-\infty,\xi_{\max}(v)\bigr)$ containing $0$, and such that:
\[
a(0)=\frac{1}{2}
\,.
\]
According to the statements above and to the local study around the equilibrium $(1,0,0)$ (\cref{sec:neighb_equil}), one of the three (mutually exclusive) cases occurs for this solution.
\begin{enumerate}
\item The amplitude $a_v(\xi)$ approaches $+\infty$ when $\xi$ approaches $\xi_{\max}(v)$.
\item The solution is defined up to $+\infty$ and approaches the equilibrium $(1,0,0)$ through its stable manifold, when $\xi$ approaches $+\infty$.
\item The solution is defined up to $+\infty$ and approaches the equilibrium $(1,0,0)$ through its center manifold, when $\xi$ approaches $+\infty$.
\end{enumerate}
Let us denote by $\vinfty$, $\vstab$, and $\vcent$ the subsets of $\rr$ containing the values of $v$ corresponding respectively to these three scenarios. We are going to show that the conclusions of \cref{thm} hold with:
\begin{equation}
\label{not_trichot_alt}
\vnone = \vinfty
\quad\mbox{and}\quad
\vSteep = \vstab
\quad\mbox{and}\quad
\vGradual = \vcent
\,.
\end{equation}
(see \cref{fig:partition_velocity_line_bis}).
\begin{figure}[!htbp]
	\centering
    \includegraphics[width=0.9\textwidth]{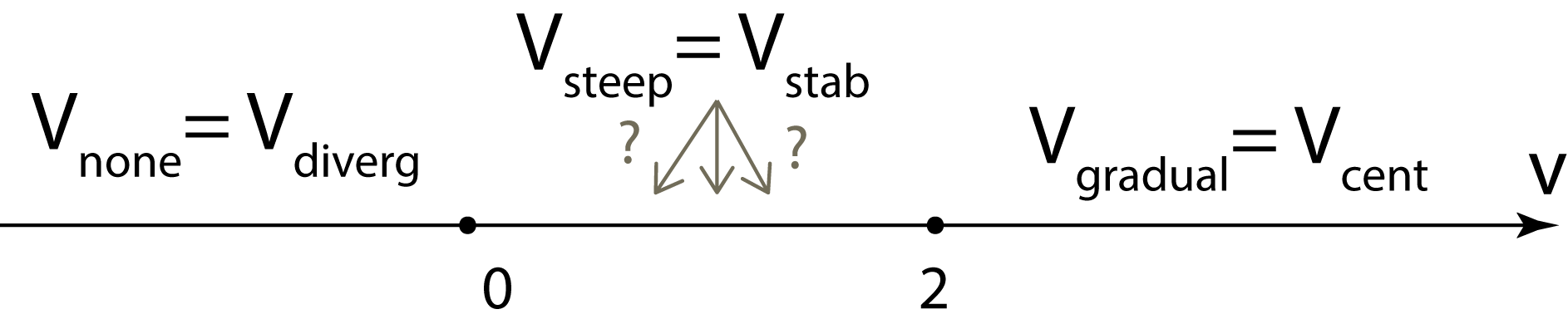}
    \caption{Partition of the velocity line.}
    \label{fig:partition_velocity_line_bis}
\end{figure}

We have
\[
\rr = \vinfty \sqcup \vstab \sqcup \vcent
\]
and, according to \cref{lem:a_mon} (or to the statements above),
\[
(-\infty,0]\subset \vinfty
\,.
\]
\begin{lemma}[the set $\vinfty$ is open]
The set $\vinfty$ is open in $\rr$. 
\end{lemma}
\begin{proof}
Observe that if $a_v(\xi)$ is larger than $1$ for a certain value of $\xi$, then $v$ must belong to $\vinfty$. The statement follows from the unstable manifold manifold theorem (continuity of the unstable manifold of $z_-(v)$ with respect to the parameter $v$).
\end{proof}
\begin{lemma}[the set $\vcent$ is open]
The set $\vcent$ is open in $\rr$. 
\end{lemma}
\begin{proof}
This statement follows from the continuity of the unstable manifold of $z_-(v)$ and of the (local) stable manifold of $(1,0,0)$ with respect to the parameter $v$.
\end{proof}
\begin{lemma}[the set $\vcent$ contains the interval $[2,+\infty)$]
Every quantity $v$ that is not smaller than $2$ belongs to $\vcent$.
\end{lemma}
\begin{proof}
Let us consider the function
\[
\qqq:\rr\times\cc\rightarrow\rr, 
\quad
(a,z)\mapsto \frac{\abs{z}^2-(1-a)^2}{2}
\]
and the set 
\[
\ccc = \{ (a,z)\in\rr\times\cc : 0<a<1 \mbox{ and } \qqq(a,z)<0 \}
\,.
\]
This set is the right circular open solid cone of apex $(1,0_{\cc})$ and base the disc of center the origin and radius $1$ in the plane $\{(a,z)\in\rr\times\cc:a=0\}$ (see \cref{fig:attractive_cone}).
\begin{figure}[!htbp]
	\centering
    \includegraphics[width=0.5\textwidth]{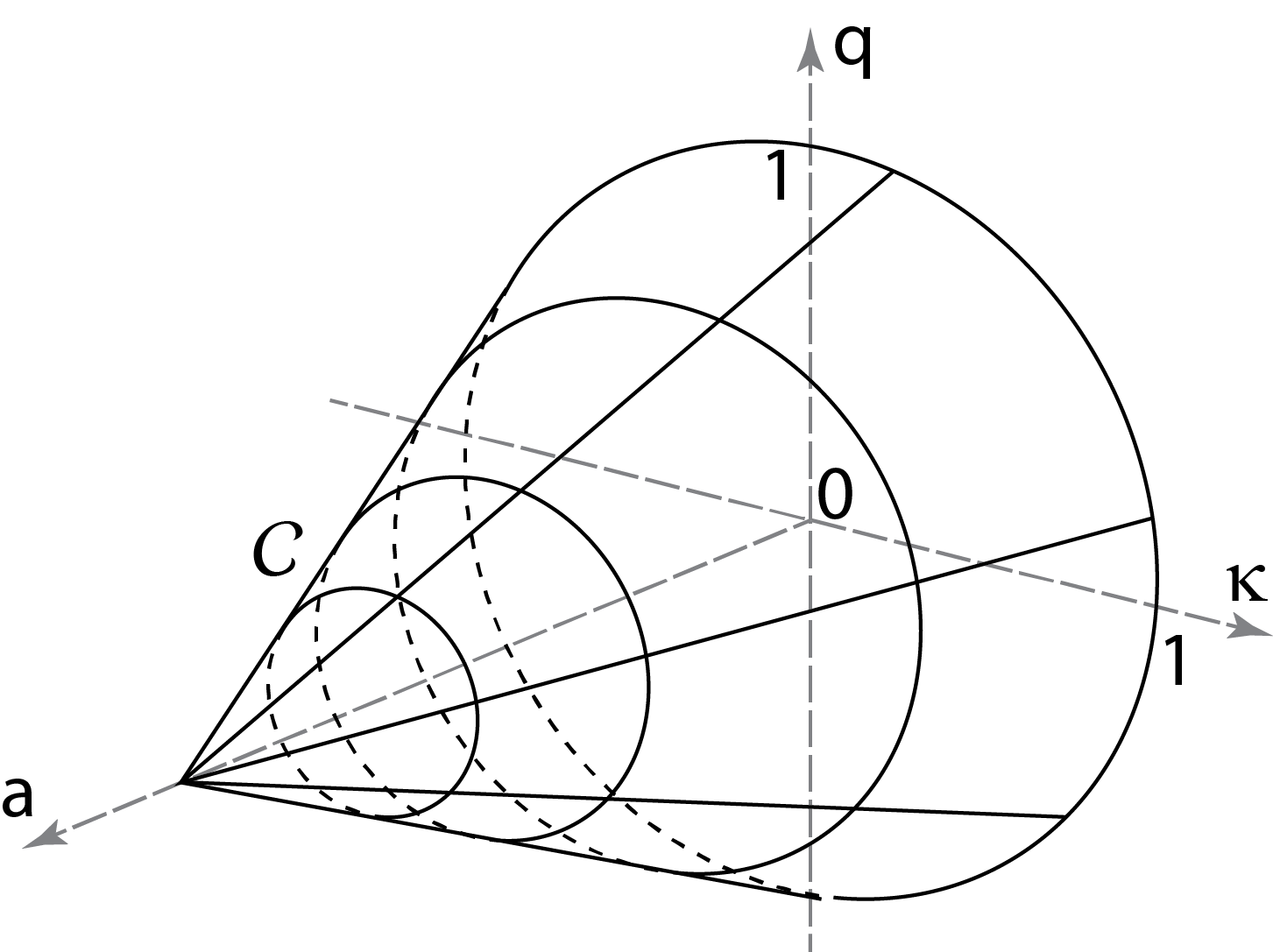}
    \caption{The cone $\ccc$ (positively invariant for the flow of system \cref{syst_az_om_1} as soon as $v$ is not smaller than $2$).}
    \label{fig:attractive_cone}
\end{figure}

Assume that the velocity $v$ positive. Then, according to \cref{abs_z_plus}, $\abs{z_+(v)}$ is smaller than $1$, thus $\qqq\bigl(0,z_+(v)\bigr)$ is negative, and as a consequence, if we write 
\[
z_v(\xi)=\kappa_v(\xi)+iq_v(\xi) 
\quad\mbox{for all}\quad
\xi\mbox{ in }\bigl(-\infty,\xi_{max}(v)\bigr)
\,,
\]
then $\bigl(a_v(\xi),z_v(\xi)\bigr)$ belongs to $\ccc$ for $\xi$ large negative. 

Let us assume furthermore that $v$ is not smaller than $2$. We are going to prove that, $\bigl(a_v(\xi),z_v(\xi)\bigr)$ actually remains in $\ccc$ (in other words that $\qqq\bigl(a_v(\xi),z_v(\xi)\bigr)$ remains negative) for all $\xi$ in $\bigl(-\infty,\xi_{\max}(v)\bigr)$. 

Let us proceed by contradiction and assume that there exists $\xi_0$ in $\bigl(-\infty,\xi_{\max}(v)\bigr)$ such that 
\begin{equation}
\label{escape_cone}
\qqq\bigl(a_v(\xi_0),z_v(\xi_0)\bigr)=0
\quad\mbox{and}\quad 
\qqq\bigl(a_v(\xi),z_v(\xi)\bigr)<0
\quad\mbox{for all}\quad
\xi\in(-\infty,\xi_0)
\,.
\end{equation}
For every solution $\xi\mapsto\bigl(a(\xi),z(\xi)\bigr)$ of system \cref{syst_az_om_1}, a direct computation gives
\[
\frac{d}{d\xi}Q\bigl(a(\xi),z(\xi)\bigr) = (1-a^2) \imm z - v\abs{z}^2  - \abs{z}^2 \ree z - a (1-a) \ree z
\]
thus, since $\ree z_v(\xi)$ takes solely positive values,
\[
\frac{d}{d\xi}_{|\xi=\xi_0} Q\bigl(a_v(\xi),z_v(\xi)\bigr) < 2 \bigl(1-a_v(\xi_0)\bigr)\imm z_v(\xi_0)- v\abs{z_v(\xi_0)}^2
\]
and since (according to \cref{escape_cone}) $1-a_v(\xi_0)=\abs{z_v(\xi_0)}$, this yields
\[
\frac{d}{d\xi}_{|\xi=\xi_0} Q\bigl(a_v(\xi),z_v(\xi)\bigr) < (2-v)\abs{z_v(\xi_0)}^2\le 0
\,,
\]
a contradiction with \cref{escape_cone}.

As a consequence (still assuming that $v$ is not smaller than $2$), the quantity $\xi_{\max}(v)$ equals $+\infty$ and $\bigl(a_v(\xi),\kappa_v)\xi),q_v(\xi)\bigr)$ approaches $(1,0,0)$ when $\xi$ approaches $+\infty$. But this approach cannot occur through the stable manifold of $(1,0,0)$ or else, according to the expression \cref{lin} of the linearization of the system at $(1,0,0)$, it would occur tangentially to the direction of vector $(0,0,1)$, which is impossible while remaining in $\ccc$. Thus $v$ belongs to $\vcent$. 
\end{proof}
\begin{lemma}[for every retracting front of amplitude $1$ the functions $a_v(\cdot)$ and $\kappa_v(\cdot)$ and $q_v(\cdot)$ are positive]
\label{lem:q_pos}
For all $v$ in $\vstab\sqcup\vcent$, the trajectory of the solution $\xi\mapsto\bigl(a_v(\xi),\kappa_v(\xi),q_v(\xi)\bigr)$ belongs to the octant:
\[
\{(a,\kappa,q)\in\rr^3 : a>0 \mbox{ and }\kappa>0 \mbox{ and } q>0 \}
\,.
\]
\end{lemma}
\begin{proof}
We already know that $a_v(\xi)$ and $\kappa_v(\xi)$ remain positive for all $\xi$ in $\rr$. Since $a_v(\xi)$ remains smaller than $1$ and since $q_v(\xi)$ is positive for $\xi$ large negative, the expression of $dq/d\xi$ in \cref{syst_az_om_1} shows that $q_v(\xi)$ actually remains positive for all $\xi$ in $\rr$. 
\end{proof}
The monotonicity of the phase (\vref{prop:monotonic_phase}) follows from this lemma. The remaining assertions of \cref{thm} follow from the asymptotics \cref{approach_stab}, \cref{tot_phase_stab}, \cref{approach_cent}, and \cref{tot_phase_center} about the approach to $(1,0,0)$ through its stable or center-stable manifold. \Cref{thm} is proved.

\subsubsection*{Acknowledgements} 
We thank Pierre Coullet and Lorenz Kramer for explaining us their work and the arisen questions. The numerical computations (\cref{fig:shape_front} and approximate velocity of the front) were performed by Pierre Coullet using the NLKit software developed at the Institut Non Linéaire de Nice. 

\printbibliography % Biblio

%\listoftodos % Lister les todos. 

\end{document}